\documentclass[11pt,draftcls,onecolumn]{IEEEtran}
\usepackage[cmex10]{amsmath}

\usepackage{amssymb}

\usepackage{cite}
\usepackage{txfonts}

\usepackage{enumitem}
\usepackage{dsfont}
\usepackage{graphicx}
\usepackage[usenames, dvipsnames]{color}
\usepackage{tikz}
\usetikzlibrary{automata, positioning, shapes}
\usepackage{mathtools}
\allowdisplaybreaks
\newtheorem{theorem}{Theorem}
\newtheorem{corollary}{Corollary}
\newtheorem{lemma}{Lemma}
\newtheorem{proposition}{Proposition}
\newtheorem{definition}{Definition}

\newtheorem{prob}{Problem}
\newtheorem{remark}{Remark}
\newtheorem{example}{Example}

\DeclareMathOperator{\col}{col}
\DeclareMathOperator{\rev}{rev}

\newcommand{\re}{{\mathbb R}}

\newcommand{\ree}{{\mathbb R}}

\newcommand{\imag}{{\rm Im  }}
\newcommand{\nat}{{\mathbb N}}
\newcommand{\n}{{\mathbb N}}

\newcommand{\cA}{{\cal{A}}}

\newcommand{\cL}{{\cal{L}}}
\newcommand{\cW}{{\cal{W}}}
\newcommand{\dmax}{{d_{max}}}

\newcommand{\cD}{{\cal{D}}}

\providecommand{\com[2]}{\begin{tt}[#1: #2]\end{tt}}

\providecommand{\mh[1]}{{\color{black}#1}}
\providecommand{\rmj[1]}{{\color{black}#1}}
\providecommand{\rmjj[1]}{{\color{black}#1}}
\providecommand{\rmjjj[1]}{{\color{black}#1}}

\newcommand{\pmat}[1]{\begin{pmatrix}#1\end{pmatrix}}
\newcommand{\lA}{\mathcal{L}(\cA)}
\newcommand{\syso}{(A,C,\cA)}
\newcommand{\sysc}{(A,B,\cA)}

\providecommand{\akc[1]}{{\color{black}#1}}

\begin{document}

\title{{Observability and controllability analysis of linear systems subject to data losses}}

\author{Rapha\"el M. Jungers, Atreyee Kundu, W.P.M.H. (Maurice) Heemels%
	\thanks{R. M. Jungers is with the Cyber-Physical Systems Laboratory at the University of California, Los Angeles (on sabbatical leave from UCLouvain, Belgium). A. Kundu and W.P.M.H. Heemels are with the Control Systems Technology Group, Department of Mechanical Engineering, Eindhoven University of Technology, 5600 MB Eindhoven, The Netherlands. Emails: \texttt{raphael.jungers@uclouvain.be, a.kundu@tue.nl, m.heemels@tue.nl}.}
	\thanks{Raphael is supported by the Communaut\'e francaise de Belgique, by the Belgian Programme on Interuniversity Attraction Poles, and the Fulbright Commission. He is a F.R.S.-FNRS Research Associate.  }
	\thanks{Maurice is supported by the Innovational Research Incentives Scheme under the VICI grant ``Wireless control systems: A new frontier in automation'' (no. 11382) awarded by NWO (Netherlands Organisation for Scientific Research), and STW (Dutch Science Foundation), and the STW project 12697 ``Control based on data-intensive sensing.''}
}

\markboth{}{}

\maketitle

\begin{abstract}
   We provide algorithmically verifiable necessary and sufficient conditions for fundamental system theoretic properties of discrete-time linear systems subject to data losses.  More precisely, the systems in our modeling framework are subject to disruptions (data losses) in the feedback loop, where the set of possible data loss sequences is captured by an automaton.
    As such, the results are applicable in the context of shared (wireless) communication networks and/or embedded architectures where some information on the data loss behaviour is available a priori. 
		
		We propose an algorithm for deciding observability (or the absence of it) for such systems, and show how this algorithm can be used also to decide other properties including constructibility, controllability, reachability, null-controllability, detectability and stabilizability by means of relations that we establish among these properties.     The main apparatus for our analysis is the celebrated Skolem's Theorem from linear algebra. \\Moreover, we study the relation between the model adopted in this paper and a previously introduced model where, {instead of allowing dropouts in the feedback loop, one allows for time-varying delays.}
\end{abstract}

\IEEEpeerreviewmaketitle
\section{Introduction}
\label{s:introduction}
	
The last decades have witnessed the introduction of many sensing, computing and wireless communication devices in a broad range of control applications leading to networked control and cyber-physical systems. In such systems it is observed that the phenomenon of data losses is rather common as a real-time feedback loop can be disrupted intermittently by packet losses in wireless communication, task deadlines can be missed in shared embedded processors, outliers can be discarded in sensor-harvested data, and so on.  This has caused a vivid interest in the analysis and design of control systems, see, e.g., \cite{Sinopoli_04, Tabbara_07, Pajic_11, D'Innocenzo_13, Gommans_13, Bemporad_ncs, Schenato_07, Smith_03} and the references therein.

As data losses lead to imperfect control updates, it can be expected that fundamental system theoretic properties such as observability, controllability, detectability, stabilizability, etc. are affected. These fundamental notions have played central roles throughout the history of modern control theory and have been studied extensively in the context of finite-dimensional linear systems, nonlinear systems, infinite-dimensional systems, n-dimensional systems, hybrid systems, and behavioral systems. One may refer, for instance, to \cite{Sontag_98} for historical comments and references. Given the importance of these notions on the one hand and the relevance of cyber-physical and networked control systems on the other hand, we believe that observability (and related notions) under packet loss will be critical in forthcoming engineering applications.

Figures \ref{fig:obsv_loss} and \ref{fig:contr_loss} illustrate two situations of intermittent data losses that are of interest and will be studied in the present paper. It is rather remarkable that these questions have not been answered despite the large literature on systems with packet loss.  The above-mentioned references have introduced (analysis and design) methods that can guarantee stability of the closed-loop systems (up to some conservativeness), but to our knowledge the fundamental properties in dropout systems have not been studied from a theoretical and algorithmic point of view, as they have been for linear (and, later, for nonlinear) systems. 

\begin{figure}[htbp]
        \begin{center}
            \begin{tikzpicture}[every path/.style={>=latex},base node/.style={double,draw,rectangle,thick}, real node/.style={draw, rectangle,scale = 0.7}, complex node/.style={draw,rectangle,scale = 0.9}]
            \node[complex node]   (a) at (-3,0)  { Plant };
            \node[complex node]   (e) at (2.8,0) { Observer };
            \node[cloud, cloud puffs=12.7, minimum width=2cm, draw] (c) at (0,0) {Network};

            \draw[->] (-4,0) -- (a);
            \draw[->] (e) -- (4,0);
            \draw[->] (a) -- (-1.2,0);
            \draw[->] (1.16,0) -- (e);
        \end{tikzpicture}
        \caption{Network-induced imperfections in the form of intermittent loss of plant output (Section \ref{s:obsv_res}).} \label{fig:obsv_loss} 
        \end{center}
  \end{figure}
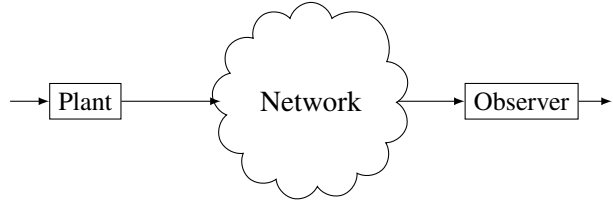

  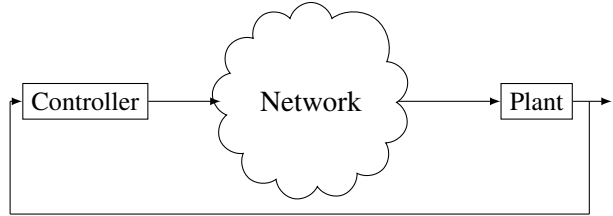
\begin{figure}[htbp]
        \begin{center}
            \begin{tikzpicture}[every path/.style={>=latex},base node/.style={double,draw,rectangle,thick}, real node/.style={draw, rectangle,scale = 0.7}, complex node/.style={draw,rectangle,scale = 0.9}]
            \node[complex node]   (a) at (-3,0)  { Controller };
            \node[complex node]   (e) at (3,0) { Plant };
            \node[cloud, cloud puffs=12.7, minimum width=2cm, draw] (c) at (0,0) {Network};

            \draw (3.7,0) edge (3.7,-1.5);
            \draw (3.7,-1.5) edge (-4,-1.5);
            \draw (-4,-1.5) -- (-4,0);
            \draw[->] (-4,0) -- (a);
            \draw[->] (e) -- (4,0);
            \draw[->] (a) -- (-1.2,0);
            \draw[->] (1.16,0) -- (e);
        \end{tikzpicture}
        \caption{Network-induced imperfections in the form of intermittent loss of control values (Section \ref{s:contr_res}). } \label{fig:contr_loss}
        \end{center}
  \end{figure}
  
Motivated by these considerations, we will address in this paper the problem of studying observability, controllability, detectability, stabilizability and other fundamental properties for discrete-time linear time-invariant (LTI) systems with packet losses represented in Figures \ref{fig:obsv_loss} and \ref{fig:contr_loss}. We will also study the relationships between these notions. We employ a constrained switched systems perspective for modelling the system under consideration. A constrained switched system is a switched system in which the admissible switches are constrained, e.g., it is allowed to switch from mode \(a\) to mode \(b\), but it is not allowed to switch from mode \(b\) to mode \(a\), etc. More precisely, the admissible switches are described by an automaton (see Definition \ref{d:automaton} below for the details). These \emph{constrained switched linear systems} have received considerable attention in recent years \cite{DaAGTS,philippe2015stability,WaRoSOLA,KoTBWF}, and many tools or results that were initially designed for general switched systems are now available for constrained switched systems.

In our setting the switching constraints arise from the consideration that the data loss signal has some known properties that prevent arbitrary switches. More precisely, the language generated by the automaton represents the admissible data loss signals. This setup encapsulates many natural constraints, including the case where there is an upper bound on the maximal number of consecutive dropouts and many others. 

Although for certain classes of switched systems and hybrid systems, questions of controllability, observability and others have been addressed, see, e.g., \cite{Babaali_05, Camlibel_08, Camlibel_07, Sun_02, Kaba_14, Bemporad_00, Xie_03}, there are no general answers to these problems, certainly not in the context of constrained switched linear systems as considered here. Indeed, the complexity of characterizing controllability has been studied by \cite{Blondel_99} for some classes of hybrid systems, and these authors showed that even within quite limited classes of switched systems, there is no algorithm to decide the controllability properties of a given system. However, by exploiting the specific structure of the systems described above, and combining this with deep results from linear algebra and arguments from automata theory, we will be able to answer the questions at hand. 

We build hereby on our preliminary conference paper \cite{Jungers_obsv15} where we tackle the single \emph{observability} problem (under additional restrictions). The analogous problem of \emph{controllability} was studied by us in \cite{Jungers2015controllability}. In the present journal version, we extend these results to various other system theoretic properties, and provide a complete picture of the problem. In particular, we study observability, constructibility, detectability, controllability, reachability, null-controllability and stabilizability of discrete-time linear systems subject to data losses (and their practical variants). We show that all these properties can indeed be verified algorithmically and also reveal various duality relationships between the mentioned notions (and, interestingly, the \emph{absence} of duality in some situations, contrary to the LTI case). In addition, we connect to the recent work in  \cite{JUNGERS_12} in which  the authors studied controllability algorithms for another class of systems, containing switching delays in the feedback loop.  We will study extensively the link between that model and ours in Section \ref{s:swdelay_connxn} below.

The remainder of this article is organized as follows. In Section \ref{s:probformu}, we present our model in details.  Then, in Section \ref{s:obsv_res}, we analyze the observability problem (and the similar constructibility problem) and in Section \ref{s:contr_res}, we analyze the controllability counterpart: the reachability and $0$-controllability problems.  In Section \ref{s:stab_detec_res}, we further generalize the technique to the stabilizability and detectability problems.  Our last theoretical contribution comes in Section \ref{s:swdelay_connxn}, where we take a step back and analyze the connections between our problem and the recently studied controllability with switching delays.  In Section \ref{s:num_ex}, we give a numerical example, and then finish with a brief conclusion.

\section{Problem formulation}
\label{s:probformu}
	In this paper we are primarily interested in studying the following systems and properties:
	\begin{itemize}
		\item Observability of discrete-time linear systems subject to data losses of the form
		\begin{align}
		\label{e:swsys1}
		\begin{aligned}
			x(t+1) &= 	Ax(t),\\
			y(t) &= \begin{cases}
			Cx(t),\:\:&\text{if}\:\sigma(t) = 1,\\
			 \emptyset,\:\:&\text{if}\:\sigma(t) = 0,
			\end{cases}
		\end{aligned}
	\end{align}
	where $x(t)\in\re^{n}$, and $y(t)\in\re^{p}$ are respectively the vectors of states and outputs at time $t\in\nat,$ and $\sigma: \nat\to\{0,1\}$ is the \emph{data loss signal} in the sense that $\sigma(t) = 1$ corresponds to the case where the output is available at time $t\in\nat$, and $\sigma(t) = 0$ corresponds to the case where the output is lost at time $t\in\nat$.
		\item Controllability of discrete-time linear systems subject to data losses of the form
		\begin{align}
		\label{e:swsys2}
		\begin{aligned}
			x(t+1) &= \begin{cases}
				Ax(t) + Bu(t),\:\:&\text{if}\:\sigma(t) = 1,\\
				Ax(t),\:\:&\text{if}\:\sigma(t) = 0,
			\end{cases}		
		\end{aligned}
	\end{align}
	where $x(t)\in\re^{n}$, and $u(t)\in\re^{m}$ are respectively the vectors of states and inputs at time $t\in\nat,$ and $\sigma: \nat\to\{0,1\}$ is the \emph{data loss signal} in the sense that  $\sigma(t) = 1$ corresponds to the case where the input is \rmjj{effectively conveyed to the plant} at time $t\in\nat$, and $\sigma(t) = 0$ corresponds to the case where the input is lost at time $t\in\nat$.
	\end{itemize}

	{For any pair of matrices $(A,C)$ (resp. $(A,B)$), the observability (resp. controllability) problem is relevant only if some additional assumptions are made on the possible dropout behaviour (it is obvious that if $\sigma(t)=0$  for all $t\in\nat$, the system is not observable (resp.~not controllable)). For this reason, our model comes with certain constraints on the classes of data loss signals $\sigma.$} For example, a natural assumption is that there can be at most $\ell\in\nat$ consecutive data losses with $\ell$ given. Such constraints on the data loss signal are often encountered in practice, and arise from the characteristics of the shared (wireless) communication network and/or the embedded architecture. Another type of constraints could be that for every \(m\) consecutive time steps, we receive data in good order on at least \(k\) time steps (so-called $(m,k)$-firmness), where \(m,k\in\nat\) are given parameters \cite{Jia_05, Felicioni_06,horssen_ecc16}.
	
	{All the {above} constraints (and many others) on the data loss signal can be {captured} by an automaton, which is part of the description of the dynamical system. We now formally describe this notion (which slightly deviates from the literature to satisfy our needs).
	
	\begin{definition}
	\label{d:automaton}
		An \emph{automaton} is a pair  $\cA=(M,s)\in \{0,1\}^{N\times N}\times \{0,1\}^{N}$, where $N$ is the \emph{number of \rmjjj{nodes} (or states)}, $M\in \{0,1\}^{N\times N}$ is the \emph{transition matrix}, and $s = (s_1\  s_2  \ \ldots \ s_N)^\top \in \{0,1\}^{N}$ is the \emph{vector of node labels}.
	\end{definition}
		\rmjj{
	\begin{definition}
	\label{d:adm_swsig}
		A data loss signal $\sigma:\, \n \rightarrow \{0,1\}$ is said to be \emph{admissible} with respect to an automaton $\cA=(M,s)$ with $N$ nodes, if there exists a sequence of states $v:\n \rightarrow \{1,2,\dots, N\}$ such that for all  $t \in\nat$ it holds that $M_{v(t),v(t+1)}=1$ and $\sigma(t)=s_{v(t)}$. The set of all admissible data loss signals for the automaton $\cA$ is denoted by $\lA$.	
	\end{definition}}
	In this paper we consider only automata that can produce at least one admissible data loss signal (and thus by definition defined for all $t\in\nat$), which implies that the automaton must contain at least one {\em cycle}. 
	\begin{example}
	\label{ex-maxdropouts}
		Consider a control system where there can be at most $\ell\in\n$ consecutive data losses. This can be captured by an automaton containing $\ell+1$ nodes, in which the node $i\in\{1,2,\cdots,\ell\}$ represents {the situation where the data generated at the} last $i$ consecutive instants were lost, and the one before arrived safely.  The node $\ell+1$ represents {the situation where the} last packet arrived safely. Let $\ell=3$. The corresponding automaton $\cA = (M,s)$ with $M=
 	\begin{pmatrix}
 		0 & 1 & 0 & 1 \\
 		0 & 0 & 1 & 1 \\
 		0 & 0 & 0 & 1 \\
 		1 & 0 & 0 & 1
 	\end{pmatrix}$ and $s=
		\begin{pmatrix}
 			0 \\ 0 \\ 0 \\ 1
 		\end{pmatrix}$ is shown in Figure \ref{fig:automat_ex}.
		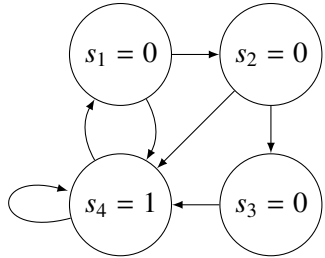
\begin{figure}[h!]
\begin{center}
        \begin{tikzpicture}[every path/.style={>=latex},every node/.style={auto}]
            \node[state]            (a) at (-1,-1)  {$s_4=1$ };
            \node[state]            (b) at (1,-1)  { $s_3=0$ };
            \node[state]            (c) at (-1,1) { $s_1=0 $};
            \node[state]            (d) at (1,1)  {$s_2=0$};

            \draw[->] (a) edge[bend left] (c);
            \draw[->] (c) edge[bend left] (a);
            \draw[->] (a) edge[loop left] (a);
            \draw[->] (b) edge (a);
            \draw[->] (d) edge (b);
             \draw[->] (c) edge (d);
              \draw[->] (d) edge (a);
        \end{tikzpicture}
        \end{center}
\caption{Automaton representing data loss signals with at most $3$ consecutive losses. Labels are represented directly on the nodes.}
\label{fig:automat_ex}
\end{figure}
\end{example}
Combining the description of the dynamics \eqref{e:swsys1} (or \eqref{e:swsys2}) and the automaton $\cA$, our systems can be modeled as so-called \emph{discrete-time constrained switched linear systems}. These systems, where the switching signal is constrained by an automaton, have been the subject of much attention recently \cite{DaAGTS,philippe2015stability,WaRoSOLA,KoTBWF}, see also the introduction. 

	In the sequel we refer to the discrete-time switched systems \eqref{e:swsys1} and \eqref{e:swsys2} constrained by an automaton $\cA$ with the tuples $\syso$ and $\sysc$, respectively.
	Similarly, the notations $x(t,x_{0},\sigma)$ and $y(t,x_{0},\sigma)$ denote the state and output trajectories of the system  \eqref{e:swsys1}  at time $t\in\nat$ generated under a switching signal $\sigma$ with initial state $x_{0}$.
The notation $x(t,x_{0},u, \sigma)$  denotes the state and output trajectories of the system  \eqref{e:swsys2}  at time $t\in\nat$ generated under a switching signal $\sigma$ and control input signal $u:\nat\rightarrow \ree^m$  with initial state $x_{0}$.
 	 However, when there is no ambiguity, we stick to $x(t)$ and $y(t)$ for brevity.

	\begin{definition}
	\label{d:all-defn1}
		(i) The system  $\syso$ is said to be \emph{observable} if for any $\sigma\in\lA$, any pair of initial states $x_{0}$ and $\tilde{x}_{0}\in\re^{n}$, it holds that $y(t,x_{0},\sigma) = y(t,\tilde{x}_{0},\sigma)$ for all $t\in\nat$ implies that $x_{0}=\tilde{x}_{0}$. 
		The system is called \emph{unobservable} otherwise.\\\\
		(ii) The system $\syso$ is said to be \emph{constructible} if for any $\sigma\in\lA$, and any initial state $x_{0}$, there exists a $T\in\nat$ such that for any $\tilde{x}_{0}\in\re^{n}$ that satisfies $y(t,x_{0},\sigma) = y(t,\tilde{x}_{0},\sigma)$ for all $t\in\nat_{[0,T]}$ it holds that $x(T,x_{0},\sigma) = x(T,\tilde{x}_{0},\sigma)$.
		The system is called \emph{non-constructible} otherwise.\\\\
		(iii) The system $\syso$ is said to be \emph{detectable} if for any $\sigma\in\lA$ and any initial state $x_{0}\in\re^{n}$ with $y(t,x_{0},\sigma) = 0$ for all $t\in\nat,$ it holds that $x(t,x_{0},\sigma)\rightarrow 0$ as $t\rightarrow \infty$. The system is called \emph{undetectable} otherwise.
	\end{definition}
	
	\begin{definition}
	\label{d:all-defn2}
		(i) The system $\sysc$ is said to be \emph{controllable} if for any $\sigma\in\lA$, any initial state $x_{0}\in\re^{n}$ and any final state $x_{f}\in\re^{n}$, there is an input signal $u:\nat\rightarrow\re^m$ such that $x(T,x_{0},u,\sigma) = x_{f}$ for some $T\in\nat$. The system is called \emph{uncontrollable} otherwise.\\\\
		(ii) The system $\sysc$ is said to be \emph{reachable} if for any $\sigma\in\lA$, and for the initial state $0\in\re^{n}$ and any final state $x_{f}\in\re^{n}$, there is an input signal $u:\nat\rightarrow\re^m$ such that $x(T,x_{0},u,\sigma) = x_{f}$ for some $T\in\nat$. The system is called \emph{unreachable} otherwise.\\\\
		(iii) The system $\sysc$ is said to be \emph{0-controllable} if for any $\sigma\in\lA$ and any initial state $x_{0}\in\re^{n}$, there is an input signal $u:\nat\rightarrow\re^m$ such that $x(T,x_{0},u,\sigma) = 0$ for some $T\in\nat$. The system is called \emph{non-0-controllable} otherwise.\\\\
		(iv) The system $\sysc$ is said to be \emph{stabilizable} if for any $\sigma\in\lA$ and any initial state $x_{0}\in\re^{n}$,  there is an input signal $u:\nat\rightarrow\re^m$ such that $\lim_{t\rightarrow\infty} x(t,x_{0},u,\sigma) = 0$. The system is called \emph{non-stabilizable} otherwise.
	\end{definition}


To motivate the study of these properties, let us consider an example.

\begin{example}
\label{ex:unobsv}
    Consider the system \eqref{e:swsys1} with
	\[
		A = \pmat{0 & 1\\1 & 0}\:\:\text{and}\:\:C = \pmat{0 & 1}.
	\]
	The observability matrix \eqref{e:swsys1} (in the classical sense, that is, with \(\sigma(t) = 1\) for all \(t\in\nat\)) is given by
		\[
		O(A,C) =  \mh{\pmat{C \\ CA}= \pmat{0 & 1\\1 & 0}}.
	\]
	Notice that the rank of \(O(A,C)\) is 2, and consequently, the pair \((A,C)\) is observable in the classical sense.\\
	\rmjjj{Let us now suppose that the system is subject to dropout signals, as described by the automaton $\cA$ in Figure \ref{fig:automat_ex}. Under the periodic data loss signal \(\sigma = (1,0,1,0,1,0,\ldots) \), which is admissible with respect to $\cA$,} it is easy to see that the system is no longer observable. \\ 
	Now, consider the situation where out of every four consecutive transmissions, at least three are received. \akc{This situation is depicted in Figure \ref{fig:auto_m,k}, where we have an automaton \(\cA = (M,s)\) with \(M = \pmat{0 & 1 & 0 & 0\\0 & 0 & 1 & 0\\0 & 0 & 0 & 1\\1 & 0 & 0 & 1}\) and \(s = \pmat{0\\1\\1\\1}\).} 
In this case, it is easy to see that the observability property would be conserved.
\begin{figure}[h!]
\begin{center}
        \begin{tikzpicture}[every path/.style={>=latex},every node/.style={auto}]
            \node[state]            (a) at (-1,-1)  {$s_4=1$ };
            \node[state]            (b) at (1,-1)  { $s_3=1$ };
            \node[state]            (c) at (-1,1) { $s_1=0 $};
            \node[state]            (d) at (1,1)  {$s_2=1$};

            \draw[->] (a) edge[loop left] (a);
            \draw[->] (b) edge (a);
            \draw[->] (d) edge (b);
             \draw[->] (c) edge (d);
              \draw[->] (a) edge (c);
        \end{tikzpicture}
        \end{center}
\caption{Automaton representing data loss signals where out of every 4 consecutive transmissions, at least 3 are received.}
\label{fig:auto_m,k}
\end{figure}
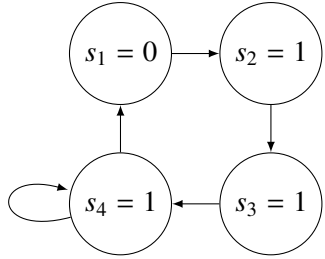
\end{example}


Interestingly, we can characterize the observability and constructibility notions in simpler forms, which will be beneficial later. This can be accomplished by noting that for any $\sigma \in \lA$ fixed, it holds that the systems \eqref{e:swsys1} and \eqref{e:swsys2} are linear (though time-varying). Hence, due to linearity, it holds that $y(t,x_0,\sigma)-y(t,\tilde x_0,\sigma) = y(t,x_0 - \tilde x_0,\sigma)$ and $x(t,x_0,\sigma)-x(t,\tilde x_0,\sigma) = x(t,x_0 - \tilde x_0,\sigma)$. As a consequence, \\\\ (i') {\em the system $\syso$ is {observable} if  and only if for any $\sigma\in\lA$ and any initial state $x_{0}$, it holds that $y(t,x_{0},\sigma) = 0$ for all $t\in\nat$ implies that $x_{0}={0}$.} \\\\
Similarly,  \\\\ (ii') {\em the system $\syso$ is constructible  if and only if for any $\sigma\in\lA$  there exists a $T\in\nat$ such that for any ${x}_{0}\in\re^{n}$ that satisfies $y(t,x_{0},\sigma) = 0$ for all $t\in\nat_{[0,T]}$ it holds that $x(T,x_{0},\sigma) = 0$}.


\begin{remark}
	\label{rem:prac_obsv} We now introduce a variant of our notions, which might be of practical importance, where the critical time $T$ in the definitions does not depend on the particular packet loss signal, nor the initial state.\\
		The system $\syso$ is said to be \emph{practically observable} if there exists a $T\in \nat$ such that for any $\sigma\in\lA$, any pair of initial states $x_{0}$ and $\tilde{x}_{0}\in\re^{n}$ with  $y(t,x_{0},\sigma) = y(t,\tilde{x}_{0},\sigma)$ for all $t\in\nat_{[0,T]},$ it holds that  $x_{0}=\tilde{x}_{0}$. Similarly, $\syso$ is said to be \emph{practically constructible} if there exists a $T\in\nat$ such that for all $\sigma\in\lA$, and all initial states $x_{0}$ and $\tilde{x}_{0}\in\re^{n}$ with $y(t,x_{0},\sigma) = y(t,\tilde{x}_{0},\sigma)$ for all $t\in\nat_{[0,T]}$ it holds that $x(T,x_{0},\sigma) = x(T,\tilde{x}_{0},\sigma)$. Hence, in the `practical' version of the observability and constructibility notions, the time $T$ can be taken independently of $x_0$ and $\sigma$, which is obviously more practical for reconstructing initial states from the output data and the data loss signal. In fact, we will show below that the observability (resp. constructibility) property formulated in Definition \ref{d:all-defn1}-(i) (resp. (ii)) is equivalent to \emph{practical observability} (resp. \emph{practical constructibility}) as defined here. Similar observations apply to the controllability, reachability and $0$-controllability notions for which practical versions  can be defined as follows. The system $\sysc$ is said to be \emph{practically controllable} if there exists a time  $T\in\nat$ such that for any $\sigma\in\lA$, any initial state $x_{0}\in\re^{n}$ and any final state $x_{f}\in\re^{n}$, there is an input signal $u:\nat\rightarrow\re^m$ such that $x(t,x_{0},u,\sigma) = x_{f}$ for some $t\in\nat_{\leq T}.$ It is said to be \emph{practically reachable} if there exists a $T\in\nat$ such that for any $\sigma\in\lA$, and for the initial state $0\in\re^{n}$ and any final state $x_{f}\in\re^{n}$, there is an input signal $u:\nat\rightarrow\re^m$ such that $x(t,x_{0},u,\sigma) = x_{f}$ for some $t\in\nat_{\leq T}$. Finally, the system $\sysc$ is said to be \emph{practically 0-controllable} if there exists $T\in\nat$ such that for any $\sigma\in\lA$ and any initial state $x_{0}\in\re^{n}$, there is an input signal $u:\nat\rightarrow\re^m$ such that $x(t,x_{0},u,\sigma) = 0$ for some $t\in\nat_{\leq T}$.	\end{remark}

Note that all the given properties are invariant under similarity transformations. As such, we have, for instance, $(A,C,\cA)$ is observable if and only $(SAS^{-1}, CS^{-1},\cA)$ is observable, and $(A,B,\cA)$ is controllable if and only if $(SAS^{-1}, SB,\cA)$ is controllable, in which $S\in \ree^{n\times n}$ is an invertible matrix.


\section{Observability results}
\label{s:obsv_res}
	In this section we tackle the following problem:
	\begin{prob}
	\label{p:mainprob1}
		Given a constrained discrete-time switched linear system \eqref{e:swsys1} specified by the triple $\syso$, determine whether the system is (un)observable and (non-)constructible.
	\end{prob}
	We provide algorithmically verifiable necessary and sufficient conditions to the above problem thereby proving its decidability.  In fact, the observability property can be formalized in a completely algebraic way, in terms of an observability matrix, just like in the classical case.
	\begin{definition}
	\label{d:obsv_mat}
		Given a data loss signal $\sigma:\nat\rightarrow \{0,1\}$, we define the \emph{observability matrix} of the system $\syso$ for $\sigma$ at time $t\in\nat$ by
			\begin{align}
			\label{e:obsv_mat}
				O_{\sigma}(t) := \pmat{\sigma(0)C\\\sigma(1)CA\\\vdots\\\sigma(t-1)CA^{t-1}}\in\re^{tp\times n}.
			\end{align}
	\end{definition}
	\begin{proposition}
	\label{p:obsv_mat_rank}
		The system $\syso$ is observable if and only if for all $\sigma\in\lA$ there exists a $t\in\nat$ such that the observability matrix $O_{\sigma}(t)$ is of rank $n$.
	\end{proposition}
	
	\begin{proof}
		(Necessity) We employ contradiction. Suppose that the system $\syso$ is observable, but there \mh{exists a $\sigma\in\lA$ and a nonzero $\alpha\in\re^{n}$ such that for all $t$, $O_{\sigma}(t)\alpha = 0$ (note that the null space of $O_{\sigma}(t)$ is a superset of the null space of $O_{\sigma}(t+1)$ for all $t\in\nat$)}. Pick $x(0) = x_{0} = \alpha$. Then $y(t) = 0$ for all $t\in\nat$. Hence, based on (i'), the system $\syso$ is unobservable.\\
		(Sufficiency) Given that for all admissible $\sigma$, there is a $t\in\nat$ such that $O_{\sigma}(t)$ has full rank, $x(0)$ can be {\em uniquely} obtained from $\pmat{y(0)\\\vdots\\y(t)} = O_{\sigma}(t)x_{0}$. The observability of the system $\syso$  follows.
	\end{proof}
	
	Based on this proposition we sometimes call a $\sigma:\nat\rightarrow \{0,1\}$ observable if the corresponding observability matrix has full rank for some $t\in\nat$. Otherwise, we say that the data loss signal $\sigma$ is unobservable.

	Note that this proposition shows that
	\\\\ (i'') {\em the system $\syso$ is {observable} if  and only if for any $\sigma\in\lA$ there is a $T\in\nat$ such that for any initial state $x_{0}$ with $y(t,x_{0},\sigma) = 0$ for $t\in\nat_{[0,T]}$ it holds  that $x_{0}={0}$.} \\\\
	Note that the difference with respect to practical observability lies in the fact that for practical observability $T$ is independent of $\sigma \in \lA$.

	It is well known that for LTI systems, the difference between observability and constructibility lies in the (un)observability of the modes associated with zero eigenvalues.  We now generalize this result to our framework in the next lemma.
	\begin{lemma}
	\label{lem:obsv_and_construc}
	Consider the system $\syso$ and let  $S\in \ree^{n\times n} $ be a real Jordan form basis-transformation matrix such that  $$ SAS^{-1} = \begin{pmatrix} A_r&0\\0 & A_s \end{pmatrix} \text{ and } CS^{-1} = \begin{pmatrix} C_r \ C_s \end{pmatrix},   $$  where $A_r$ is regular, and $A_s$ has only zero eigenvalues.\\ Then, $\syso$ is (practically) constructible if and only if the system ${(A_r,C_r,\cA)}$ is (practically) observable.
	\end{lemma}
	\begin{proof}
	We start by recalling that $(A,C,\cA)$ is (practically) observable/constructible iff $(SAS^{-1},CS^{-1},\cA)$ is (practically) observable/constructible. As such, we can directly work  in the real Jordan basis
and we decompose the state $x\in \ree^n$ as $x=(x_r,x_s): = [ x_r^\top \ x_s^\top ]^\top$ accordingly. Similarly, we write $y_r(t,x_{r} (0),\sigma)$ as the output of the system described by $(A_r,C_r)$, with data loss signal $\sigma \in \lA$ and initial state $x_r(0)$, and $y_s(t,x_{r} (0),\sigma)$ as the output of the system given by $(A_s,C_s)$, data loss signal $\sigma \in \lA$
and initial state $x_s(0)$. Observe that due to the diagonal structure of $SAS^{-1}$ it holds that $y(t,x(0),\sigma) = y_r(t,x_{r} (0),\sigma)+y_s(t,x_{s} (0),\sigma)$ for all $t\in\nat$, where
$x(0)=(x_r(0),x_s(0))$.

(Sufficiency) Suppose first that ${(A_r,C_r,\cA)}$ is observable
and thus according to property (i'') we have that for all $\sigma \in \lA$ there is a $\tilde T(\sigma)$ such that for  any $x_{r}(0)$ with $y_r(t,x_{r} (0),\sigma)=0$ for $t\in\nat_{[0,\tilde T]}$ it holds  that $x_r(0)={0}$. To prove constructibility of $(A,C,\cA)$ take now $\sigma \in\lA$ and define $T=\tilde T(\sigma) +n.$ Suppose that $y(t,x(0),\sigma)=0$ for all $t\in\nat_{[0,T]}.$ Since for all $t\geq n,$ $x_s(t,x_s(0),\sigma) = 0,$ and thus $y_s(t,x_{s} (0),\sigma)=0,$ we have that also $y_r(t,x_{r} (0),\sigma)=0,$ and by observability of this latter system, it implies that $x_r(T)=0.$ Hence, $x(T)=(x_r(T),x_s(T))=0,$ and $\syso$ is constructible.\\
%
Note that if ${(A_r,C_r,\cA)}$ is practically observable, then the $\tilde T$ above is independent of $\sigma$ and thus also the defined $T$. Hence, we can conclude practical constructibility of $\syso$.

 (Necessity) Suppose $\syso$ is constructible.
Take any  $\sigma\in\lA$ and take $T$ as in the constructibility characterisation (ii') such that for any ${x}_{0}\in\re^{n}$ that satisfies $y(t,x_{0},\sigma) = 0$ for all $t\in\nat_{[0,T]}$ it holds that $x(T,x_{0},\sigma) = 0$. Take $x_r(0)$ arbitrary and set $x_0 = (x_r(0),0)$. Suppose $y_r(t,x_r(0),\sigma) =0$ for all $t\in\nat_{[0,T]}$. Then, we have that $y(t,x_{0},\sigma) = 0$ for all $t\in\nat_{[0,T]}$ as well (as $x_s(0)=0$) and thus according to the constructibility characterisation (ii') it holds that $x(T,x_{0},\sigma) =(A_r^T x_r(0), 0)= 0$. Using invertibility of $A_r^T$ yields now $x_r(0)=0$. Hence, ${(A_r,C_r,\cA)}$ is observable. The implication practical constructibility of $\syso$ $\Rightarrow$ practical observability ${(A_r,C_r,\cA)}$ follows similarly as in the sufficiency part by observing that $T$ can be taken independent of $\sigma$.

	\end{proof}

		\begin{corollary}
	\label{p:obsv_and_construc}
		If the system $\syso$ is observable, then it is constructible. Moreover, if $A$ is invertible, then constructibility and observability are equivalent.
	\end{corollary}
	\begin{proof}
		For the first part of the corollary, it follows directly from the observability characterisation (i'') and the constructibility characterisation (ii') (as $x(T,x_0,\sigma)=A^T x_0$). The second part of the corollary is obvious from Lemma \ref{lem:obsv_and_construc}.
		\end{proof}
	
	%
	%

We now turn to the main result of this section, which will allow us to decide observability.
	As announced above, we will make use of the \emph{Skolem's Theorem} \cite{Skolem_34,lech1953note} of Linear Algebra, which is given next.
	\begin{theorem}
	\label{t:skolem_thm}
	{		Consider a matrix $A\in{\mathbb R}^{n\times n}$ and two vectors $b,c\in{\mathbb R}^{n}.$} The set of values of $t\in\nat $ such that $c^\top A^{t}b = 0$ is eventually periodic in the sense that there exist two numbers $P,T\in\nat$ such that
		\begin{equation}\label{eq-thm-skolem}
			\forall t\in\nat_{>T},\:\: c^\top A^{t}b = 0\:\Leftrightarrow\:c^\top A^{t+P}b = 0.\end{equation}
		Moreover, the period $P$ can be computed explicitly.  In particular, if the entries in $A$, $b$, $c$ are rational numbers, $P$ can be chosen smaller or equal to $r^{n^2}$, where $r$ is any prime number not dividing $\det(mA)$, and $m$ is the least common multiple of the denominators of the entries in $A.$
	\end{theorem}
	%
	%
The above theorem shows that if one is interested in the patterns of times such that the response of a linear system $x(t+1)=Ax(t)$ with $x(0)=b$ is confined to a linear subspace $\ker c^\top :=\{w \in \ree^n \mid c^\top w = 0\}$,  then  one can restrict the attention eventually to periodic patterns.  

We now present our main result.  In the course of its proof, we will need the following elementary result, which connects observability and the practical versions (see Remark~\ref{rem:prac_obsv}).

	\begin{lemma}
	\label{cor-practical-obs}
				The system $\syso$ is practically observable (resp. practically constructible) if and only if it is observable (resp. constructible).
	\end{lemma}
	\begin{proof}	The `only if'-part is obvious. \mh{ For the `if'-part, let us suppose that the system is observable. We have to show that there exists a $T^*\in \nat$ such that for any $\sigma\in\lA$ and any  initial state $x_{0}\in\re^{n}$ with  $y(t,x_{0},\sigma) = 0$ for all $t\in\nat_{[0,T^*]}$ it holds that  $x_{0}=0.$  Suppose by contradiction that for every $T^*\in\nat$ there is a $\sigma^{T^*}$ with $y(t,x_{0},\sigma^{T^*}) = 0$ for all $t\in\nat_{[0,T^*]},$ but  $x_{0}\neq{0}.$ Since $$\pmat{y(0)\\ \vdots\\y(T^*)} = O_{\sigma^{T^*}}(T^*)x_{0}= {0},$$ we conclude that the observability matrix $O_{\sigma^{T^*}}(T^*)$ does not have full rank.  This implies the existence of an infinite signal $\sigma$ such that $O_{\sigma}(t)$ has never full rank, which leads to a contradiction. (The existence of $\sigma$ can be easily proved, e.g., by the convergence of the sequence $\{\sigma^{T^*}\}_{T^*\in\nat}$ w.r.t. the standard topology on words, see the proof of Lemma \ref{lem:inf_to_fin} below for more details.)\\ The equivalence between constructibility and practical constructibility follows now from Lemma~\ref{lem:obsv_and_construc}. }
	\end{proof}

	\begin{theorem}
	\label{t:mainres1}
		Problem \ref{p:mainprob1} is decidable in the sense that there is an algorithm that decides (un)observability and (non-)constructibility of  a given system specified by $\syso$ in finite time.
	\end{theorem}
	\begin{proof}
		We first present an algorithm to decide (un)observability of $\syso,$ in the particular case where the matrix $A$ is regular. We then show how it can be adapted for the singular case, and finally we show how to decide constructibility.\\
	{\bf I. Observability algorithm when $A$ is regular:}
	Our algorithm consists of two routines that run in parallel.  The first routine parses every admissible \emph{periodic} data loss signal corresponding to cycles of length $s,$ $s=1,2,\dots$ in the automaton. It then decides whether each of these cyclic signals (based on periodically repeating the cycle) renders the system observable or not (i.e., makes the corresponding observability matrix full rank or not,  which can be checked easily as we will see at the end of the proof).  If at some point a periodic signal is found, which violates the full rank condition of the corresponding observability matrix, this establishes unobservability.\\
In parallel, the second routine in our algorithm generates all admissible switching signals $\sigma$ of increasing
length and checks whether the corresponding observability matrix
$O_\sigma(t)$ has full rank. Following Proposition \ref{p:obsv_mat_rank}, if this property
is satisfied for all admissible signals of a given length, then
the system $\syso$ is observable.

The key is to prove that this algorithm is bound to stop after a finite time.
\mh{In the case that the system is observable, the second subroutine will do so. This follows from Lemma~\ref{cor-practical-obs} as observability and practical observability are equivalent. 
In the case the system is unobservable, we establish in {\bf{claim 1}} that there exists a \emph{cycle} in the automaton, whose corresponding switching signal leads to a rank-deficient observability matrix.  Thus, the first routine of our algorithm, which parses larger and larger cycles, will finally find the critical one, and the algorithm will stop.}
		
		{\bf Claim 1: proof of existence of an unobservable cyclic signal in the unobservable case.} If the system $\syso$ is unobservable, by definition, there is an admissible data loss signal $\sigma$ for which the rank of the observability matrix $O_{\sigma}(t)$ is strictly less than $n$ for all \(t\in\nat\). Hence,  the null spaces $\ker O_{\sigma}(t)$, $t\in\nat$, of the observability matrices   contain a non-trivial subspace ${\cal UO}_\sigma \neq \{ 0\}$ (one could perceive this as the subspace of states indistinguishable from zero). Fix linearly independent vectors $v_{1},\dots, v_{r}$ forming a basis of ${\cal UO}_\sigma$.
		We apply Skolem's Theorem to every pair $v_i,c_j,$ where $c_j$, $j=1,\ldots,p$, are the rows of the matrix $C$, i.e., $C=\begin{pmatrix} c_1 \\ \vdots \\ c_p\end{pmatrix}$. We obtain $T_{i,j},P_{i,j}\in\nat$, where $i=1,2\ldots,r$  and $j=1,\ldots,p$ satisfying \eqref{eq-thm-skolem} with the corresponding vector $b=v_i$ and $c^\top =c_j$.  From this, defining \begin{equation}\label{eqn-T}T=\max_{i,j}{T_{i,j}},\end{equation} and taking $P$ equal to a common multiple of the $P_{i,j},$ we have that
		%
					 					\begin{align}
		\label{eq-thm-skolem'}	
		\forall t\in\nat_{>T}, \forall_{i=1,\ldots,r}  CA^t v_i=0 \:\Leftrightarrow\: \forall_{i=1,\ldots,r} CA^{t+P} v_i=0.
			\end{align}
					
		
		Consider again the unobservable signal $\sigma$ and the corresponding path it generates in the automaton.
By the pigeonhole principle, there are two times $t_1,t_2\in\nat_{> T},$ $t_2-t_1\leq PN$ (recall that $N$ is the number of nodes in the automaton) such that \begin{equation}\label{eq-periodic}t_1=t_2 \mbox{ mod } P, \quad \mbox{and} \quad p_{\sigma}(t_1)=p_{\sigma}(t_2),\end{equation} where $p_{\sigma}(t)$ denotes the node of the infinite path corresponding to $\sigma$ in the automaton $\cA$ at time $t\in\nat$. Indeed, there are only $PN$ different possible values for the couple $(t \mbox{ mod } P,p_{\sigma}(t)).$

\mh{By \eqref{eq-periodic}, the eventually periodic signal $\tilde \sigma$ starting with $(\sigma(0),\sigma(1),\ldots, \sigma(t_1-1))$ and followed by the periodic repetition of $(\sigma(t_1),\ldots,  \sigma(t_2-1))$  corresponds to an admissible data loss signal $\tilde\sigma\in\lA$. Obviously, ${\cal UO}_\sigma \subset \ker O_{\tilde \sigma}(t_2)$ as $\sigma$ and $\tilde \sigma$ coincide for the first $t_2$ time steps.  For   $t \in \nat_{[t_2, t_2+P]}$ terms $\sigma(t) CA^t$ are added to the  observability matrices. However, we know that  $\sigma(t-P) CA^{t-P}v_i = 0$ for all $i$ and thus $\sigma(t-P)= 0$ or $CA^{t-P}v_i=0$ for all $i$. In the former case it holds that  $\sigma(t)=0$ as well due to (almost) periodicity of $\tilde \sigma$. In the latter case we can use \eqref{eq-thm-skolem'} to conclude $CA^{t}v_i=0$ for all $i$ as well. The same statement holds inductively for $t \in \nat_{\geq  t_2+P}$. Hence, all added contributions to the observability matrix satisfy also that for all $i,$ $\sigma(t) CA^{t}v_i = 0$ for $t\in \nat_{\geq t_2}$.
\\ This shows that ${\cal UO}_\sigma \subset \ker O_{\tilde \sigma}(t)$ for all $t\in \nat$ and thus that $\tilde \sigma$ makes the observability matrices rank-deficient. Moreover, also the purely periodic signal $\hat \sigma \in \lA$ obtained by the periodic repetition of the cycle $(\sigma(t_1),\ldots,  \sigma(t_2-1))$ results in rank-deficient observability matrices  (since its rows are a subset of the rows of the initial observability matrix, postmultiplied by the invertible matrix $A^{-t_1}$). }{\bf This proves claim 1.}

   Finally, in order to complete the algorithm we still have to provide a finite test for checking whether a periodic data loss signal of some period $s$ is (un)observable.\\
 For this purpose we now prove {\bf (claim 2)} that \emph{any infinite cyclic} signal $\sigma$ with period $s$ is unobservable if and only if the observability matrix $O_{\sigma}(t)$ is rank-deficient for $t=sn.$  This allows to check the unobservability of the complete signal by only looking at a finite part of it.\\ Obviously, if the signal is unobservable, the finite matrix cannot be full rank.  \\Conversely, if the matrix $O_{\sigma}(sn)$ is rank-deficient, its rows span a subspace $S$ of dimension smaller than $n.$ Consider all the values ${t_i}$ for which $\sigma(t_i)=1, t_i\leq s.$ By periodicity of $\sigma,$ all the nonzero rows in the observability matrix are of the shape $CA^{t_i+ks},k=1,2\dots.$  Now, it is easy to see that $$\mbox{row\ span}(\{CA^{t_i+ks}:k\in \n\})=\mbox{row\ span}(\{CA^{t_i+ks}:1\leq k \leq n\})$$ (apply, e.g. the Cayley-Hamilton theorem for the matrix $A^s$). Hence, $$\mbox{row\ span}(O_\sigma(t))=\mbox{row\ span}(O_\sigma(ns))\; \forall t\geq n,$$ and thus, if $O_\sigma(ns)$ is rank-deficient, the observability matrix will never become full rank. This proves claim 2.

   {\bf Summarizing the proof}, if the system is unobservable, by checking unobservability of each periodic signal corresponding to cycles of increasing length, the algorithm will come across the unobservable cycle after a finite time (as only cycles up to length $PN$ need to be verified) and the algorithm will terminate.  On the other hand, the algorithm checks every admissible switching signal of increasing length.  Thus, if the system is observable, one will come across a finite time $t$ such that all the  observability matrices $O_\sigma(t)$ have full rank (due to equivalence of observability and practical observability), and the algorithm will stop.\\

  {\bf  II. Non-invertible case:}
	
	If $A$ is not regular, the proof is rather simple. Indeed, take $x(0)\neq 0$ such that $Ax(0)=0$. For any $\sigma:\nat \rightarrow \{0,1\}$ with  $\sigma(0)=0$ it follows that $y(t,x(0),\sigma)=0$ for all $t\in\nat$, but $x(0)\neq 0$. Hence, the automaton cannot admit $\sigma \in \lA$ with $\sigma(0)=0$. This shows that observability implies in this case that all node labels are equal to $1$ and thus $\lA =\{(1,1,1,1,1,\ldots)\}$. Clearly, if $\lA =\{(1,1,1,1,1,\ldots)\}$, we have that $\syso$  is observable if and only if $(A,C)$ is observable in the classical sense, which can obviously be checked algorithmically.

%
%
%
%
%
%
%
%
%
%
{\bf III. Constructibility}
We can use Lemma~\ref{lem:obsv_and_construc} and check (non)-constructibility of $\syso$  by (un)observability of the system $(A_r,C_r,\cA),$ which can be constructed in finite time.

	%
	%
	\end{proof}
\begin{remark}\mh{
Note that Skolem's theorem not only allows to design an algorithm to decide observability properties of a system $\syso,$ but it also allows to compute a priori the running time of (a slightly modified version) of the algorithm, when one has computed an effective bound on the quantity $P$ in Skolem's theorem.  Indeed, one can simply run the second routine of the algorithm in this case as it was shown in the proof above that the system is unobservable if and only if there is an `unobservable' periodic signal based on a cycle (period) of length less than $PN$. Hence, one only has to check a finite number of cycles (of length less than $PN$) to conclude on the observability/constructability or the absence of it.  This number is bounded a priori, and one can thus deduce an a priori bound on the computational burden of the algorithm.} \mh{However, as in some cases the computation of $P$ is not so easy (or $P$ might be very large), the algorithm presented in the proof of the above theorem does not rely on the availability of $P$ by exploiting the two routines. }
\end{remark}
	
\section{Controllability results}
\label{s:contr_res}
	In this section we address the following problem:
	
	\begin{prob}
	\label{p:mainprob2}
		Given a constrained discrete-time switched linear system specified by the triple $\sysc$, determine whether the system is (un)controllable, (un)reachable and (non-)0-controllable.
	\end{prob}
	We will see that when the matrix $A$ is regular, the problem can be seen as exactly the `dual' of the observability problem, and hence solved thanks to the technique described above.  However, a bit surprisingly, the duality does not hold when $A$ is singular, and we also describe an algorithm for that case in this section.
	\begin{lemma}
	\label{lem:contr_and_0contr}
	Consider the system $\syso,$ and let $S$ be a real Jordan form basis-transformation matrix such that
	\begin{equation}\label{eq-jordan-form-contr} SAS^{-1} = \begin{pmatrix} A_r&0\\0 & A_s \end{pmatrix},\ SB = \begin{pmatrix}
	B_r \\ B_s
	\end{pmatrix},  \end{equation} where $A_r$ is regular, and $A_s$ has only zero eigenvalues.\\ Then, $\syso$ is 0-controllable if and only if ${(A_r,B_r,\cA)}$ is \mh{0-controllable}.
	\end{lemma}
	{The proof of the above lemma follows using similar arguments as in the proof of Lemma \ref{lem:obsv_and_construc} and is therefore omitted here. }
	
Just as in the observability problem, the controllability property is related with the space spanned by the columns of a \emph{reachability} matrix.
	\begin{definition}
	\label{d:reach_mat}
		Given a data loss signal $\sigma:\nat \rightarrow \{0,1\}$, we define the \emph{reachability matrix} of the system $\sysc$ at time $t\in\nat$ by
		\begin{align}
		\label{e:reach_mat}
			C_{\sigma}(t) := \pmat{ B\sigma(t-1)   & A B \sigma(t-2) & \ldots   & A^{(t-1)}B\sigma(0)  }\in \re^{n\times mt}.
		\end{align}
		Similarly, we also define the reachability matrix $C_w(t)$ at time $t\in\nat$ for finite signals (words) $w: \{0,1,\ldots, r\} \rightarrow \{0,1\}$ with $r\in\nat_{\geq t-1}$.
	\end{definition}
	
Note that 	for all $t\in \nat$ and $T\in\nat$ it holds that
	\begin{multline} C_\sigma(t+T) = \\ \left( B \sigma(t+T-1) \ AB \sigma(t+T-2) \ \ldots \ A^{T-1} B \sigma (t) \	A^T C_{\sigma}(t) \right), \label{eq:reach_lower_t}
\end{multline}
which will turn out to be a useful identity.

	\begin{proposition}
	\label{p:reach_rank}
		The system $\sysc$ is reachable if and only if for all $\sigma\in\lA$ there exists a $T\in\nat$ such that the reachability matrix $C_{\sigma}(T)$ is of rank $n$.\\		
	\end{proposition}
	\begin{proof}
		(Necessity) We proceed by contradiction. Suppose that there is a $\sigma$ such that, for all $t\in\nat,$ the controllability matrix $C_\sigma(t)$ is of rank smaller than $n.$  Let us take  $x_0=0$ as in the definition of reachability (Definition \ref{d:all-defn2}-(ii)).
From \eqref{e:swsys2} we have that \begin{equation} \label{eq-cont} x_{x_0,\sigma,u}(t)=C_\sigma(t) { u^{t-1}}, \end{equation} where $u^{t-1}\in \ree^{mt} = \col(u(t-1),u(t-2),\ldots, u(0)):= [ u(t-1)^\top \ u(t-2)^\top \ \ldots\ u(0)^\top  ]^\top$.  Hence, the set of reachable vectors at time $t\in\nat$ is a strict linear subspace of $\re^n.$ Thus, the set of reachable points, which is a countable union of these strict linear subspaces, cannot be equal to $\re^n.$

 (Sufficiency) Conversely, if for each $\sigma$ there is a $T\in\nat$ such that $C_\sigma(T)$ has full rank, then the system is clearly reachable based on \eqref{eq-cont}.
	\end{proof}
	\begin{proposition}
	\label{p:reach_and_contr} 	\label{p:reach_and_0-contr}
		The system $\sysc$ is controllable if and only if it is reachable. Moreover, reachability implies $0$-controllability and, in case $A$ is  invertible, the two notions are equivalent.
	\end{proposition}
	\begin{proof}
		The `only if' part is trivial.  Let us prove the `if' part of the first statement.
		Fix $\sigma\in\lA,$ and $x_0,x_f \in \re^n$.
		By Proposition \ref{p:reach_rank}, since the system is reachable, there is a time $t\in\nat$ such that $\text{rank}(C_{\sigma}(t)) = n$. Equivalently,
		\[
			C_{\sigma}(t)u^{t-1} = x_{f} - A^{t}x_{0}
		\]
		has a solution $u^{t-1}$ for any $x_{0}$ and $x_{f}$. The input signal $u:\nat\rightarrow \ree^m$ starting with $(u(0),u(1),\ldots,u(t-1))$ transfers the initial state $x_{0}$ to the final state $x_{f}$ at time $t$. Hence, the system $\sysc$ is controllable.\\
		\mh{
		To prove the second statement, note first that since reachability implies controllability it also implies $0$-controllability. To show the converse in case $A$ invertible, note that	given $\sigma \in \lA$ the set of all states that can be steered to the origin in finite time is given by $\bigcup_{t\in\nat} -A^{-t} \imag C_\sigma(t)$. From this identity it follows that if the system is not reachable then there exists a $\sigma$ such that $ \imag C_\sigma(t)$ is always a strict subspace of $\ree^n$. Hence, the indicated union can never be $\ree^n$ and the system is therefore not $0$-controllable. 	} 		\end{proof}
		
	We will now prove various duality relationships to connect the observability and controllability concepts. The following lemma will be useful in proving the duality relationships.

	\begin{lemma}
	\label{lem:inf_to_fin}
		Consider $\sysc$ with $A$ regular and the set of admissible signals $\cL(\cA)$ as in Definition~\ref{d:adm_swsig}.  Also, consider the set $\cW(\cA)$ of finite words generated by $\cA.$
The following two statements are equivalent:
\begin{enumerate}\item \label{cond1} There exists a signal $\sigma\in \cL(\cA)$ such that for all $t\in\nat$ the corresponding reachability matrix $C_\sigma (t)$ does not have full rank \mh{(i.e., the system is not reachable)}
\item \label{cond2}There exists a sequence $\{w_i\}_{i\in\nat_{\geq 1}}$ of finite words $w_i\in \cW (\cA)$ of length $i$ such that, for all $i\in\nat_{\geq 1}$  the corresponding reachability  matrix  $C_{w_i}(i)$ does  not have full rank \mh{(i.e., the system is not practically reachable)}.
\end{enumerate}
	\end{lemma}
		\begin{proof}
	The `I $\rightarrow$ II'-part is rather obvious since the successive prefixes of $\sigma$ given  by $w_i:=(\sigma(0),\ldots,\sigma(i-1))$ do the job as $C_{w_i}(i) = C_\sigma(i)$. 
	
%

The proof of the other direction essentially comes from compactness of the language generated by an automaton (with the standard topology on words), but we present here a self-contained proof for clarity. Therefore, suppose that there is a sequence $\{w_i\}_{i\in\nat_{\geq 1}}$ of finite words $w_i\in \cW (\cA)$ of length $i$ such that  $C_{w_i}(i)$ does not have full rank.
We now generate based on this sequence an infinite data loss signal $\sigma \in \lA$ as follows. Starting with $k=0$ we iteratively perform the following steps.
Define at step $k\in \nat$ $\sigma(k)$ such that there exists an infinite number of words in $\{w_i\}_{i\in\nat_{\geq 1}}$ with the first $k+1$ letters equal to $(\sigma(0),\sigma(1),\ldots,\sigma(k))$. This procedure can indefinitely be continued as at each step $k\in\nat_{\geq 1}$ $(\sigma(0),\sigma(1),\ldots,\sigma(k-1))$ are given (by induction) and there are an infinite number of words in $\{w_i\}_{i\in\nat_{\geq 1}}$ left with these first $k$ letters (for $k=0$ this holds trivially).
This procedure generates an infinite word $\sigma:\nat \rightarrow \{0,1\}$. Due to \eqref{eq:reach_lower_t} we obtain
\begin{multline} C_{w_i}(i) = \\ \left( B \mh{w_i}(i-1) \ AB \mh{w_i}(i-2) \ \ldots \ A^{i-k-1} B \mh{w_i} (k) \	A^{i-k} C_{\sigma}(k) \right) , \label{eq:reach_lower_t2}
\end{multline}
which holds for all $w_i$ \rmjjj{($i \in\nat_{\geq k}$)} that have the same first $k$ letters as the constructed $\sigma.$ Combining this with $C_{w_i}(i)$ not having full rank and the regularity of $A$, it follows that  for all $k\in\nat$ the corresponding observability matrix $C_{\sigma}(k)$ does not have full rank, which concludes the proof.
	\end{proof}

	We are now in position to establish the duality relationships for which we need the concept of a {\em reverse automaton}.

		\begin{definition}
	\label{d:rev_automaton}
		Let an \emph{automaton}  $\cA=(M,s)\in \{0,1\}^{N\times N}\times \{0,1\}^{N}$ be given as in Definition~\ref{d:automaton}. The reverse automaton is defined by $\cA^{r}=(M^\top,s)$.  	\end{definition}
	
	With this definition, for any $w\in\cW(\cA)$ with length $|w|$, the time-reversed word $\rev(w):\{0,1\ldots, |w|-1\}\rightarrow \{0,1\}$ given for $t\in  \{0,1\ldots, |w|-1\} $ by $\rev(w)(t):= w(|w|-t)$ satisfies $\rev(w) \in \cW(\cA^r)$. Hence, any time-reversed finite word generated by $\cA$ is a  finite word of $\cA^r$ and vice versa (as $(\cA^r)^r=\cA$).
	
	\begin{proposition}
	\label{p:obs_and_contr}
		 If the system $(A,C,\cA)$ is observable then the dual system $(A^\top,C^\top,\cA^{r})$ is controllable, where $\cA^{r}$ is the reverse automaton of $\cA$. Moreover, if $A$ is a regular matrix, then the converse statement also holds.
	\end{proposition}

	\begin{proof}
		Let us suppose that the system $\syso$ is observable and thus practically observable, and select $T$ as the bound in the definition of practical observability, i.e.,  $O_{\sigma}(T)$ has full rank for every admissible signal $\sigma$ of length $T.$
		Suppose now by contradiction that the system $(A^\top,C^\top,\cA^{r})$ is not controllable. Hence, there is a finite signal $ w\in\mathcal{W}(\cA^{r})$ of length $|w|=T$ such that the reachability matrix
\[
			C_{ w} (T)= \pmat{C^\top w (T-1) & \ldots & (A^\top)^{(T-1)}C^\top w(0)  },
		\]
		does not have full rank.
		Note that
		\begin{equation}
		\label{eq:dual} C_w(T)^\top  = O_{\rev(w)}(T)
		\end{equation} and obviously $\rev(w) \in \cW(\cA)$.  However,  $O_{\rev(w)}(T)$ has full rank because of the practical observability of system  $\syso$ with bound $T$, which is a contradiction and thus $(A^\top,C^\top,\cA^{r})$ is controllable.	
			
		To prove the second statement assume $A$ regular and $(A^\top,C^\top,\cA^{r})$ is controllable. 	Suppose now by contradiction that the system $\syso$ is not observable. Then there exists a $\sigma$ such that for all $t\in\nat$ $O_\sigma(t)$ does not have full rank. By taking the successive prefixes of $\sigma$ given  $w_i=(\sigma(0),\ldots,\sigma(i-1)),$ we obtain an infinite sequence of words $w_i\in \cW (\cA)$ of length $|w_i|=i$, $i\in\nat$, such that the corresponding observability matrix $O_{w_i}(t)$ does not have full rank for all $t\in \nat_{\leq i}$.  Using \eqref{eq:dual} and reversing the automaton, we obtain $\rev( w_i)\in \cW (\cA^r)$ such that
		$C_{\rev(w_i)}(i)$ does not have full rank. By Lemma \ref{lem:inf_to_fin}, this implies the existence of a signal $\sigma$ such that $C_\sigma(t)$ is rank-deficient for all $t \in \nat$ and thus the system is not controllable. Contradiction! Hence, $\syso$ is observable.
	\end{proof}

From now on we call $(A^\top,C^\top,\cA^{r})$ the dual system of $(A,C,\cA)$.
	
	\begin{proposition}
	\label{p:reach_and_construc}
		The system  $(A,C,\cA)$ is constructible if and only if the dual system $(A^\top,C^\top,\cA^{r})$ is 0-controllable.
	\end{proposition}
	
	\begin{proof}
	By Lemmas \ref{lem:obsv_and_construc} and \ref{p:obs_and_contr} 	we obtain $(A,C,\cA)$ is constructible $\Leftrightarrow$ $(A^r,C^r,\cA)$ observable  $\Leftrightarrow$ $((A^r)^\top ,(C^r)^\top,\cA^r)$ reachable/controllable (using that $A^r$ is regular).  Since $A^r$ is regular, reachability and 0-controllability are equivalent (combining Lemma \ref{lem:contr_and_0contr} and Proposition \ref{p:reach_and_0-contr}).
%
%
	\end{proof}
Before putting all the pieces together, we are left with the controllability problem in the case where $A$ is singular.  

\begin{example}
Take $A=0$ and $B=I$ with $I$ the identity matrix of appropriate dimensions. We take $\cA$ as the automaton that has $\lA = \{(1,0,1,0,1,0,1,\ldots), (0,1,0,1,0,1,0,1,\ldots)\}$ consisting of only periodic data loss signals with period $2$.  Note that the dual system $(A^\top,B^\top,\cA^r)$ is not observable and in fact $(0,1,0,1,0,1,0,1,\ldots)\in \cL(\cA^r) = \lA$ leads to $O_\sigma(t)=0$ not having full rank for all $t\in\nat$. However, the original system $(A,B,\cA)$ is controllable. Hence, in contrast with the LTI case without \akc{data} loss, in this case we do not have a complete duality result. Duality between controllability of a system and observability of the dual system only holds when $A$ is regular (or no \akc{data} loss is allowed).
\end{example}

This example shows that, surprisingly, the duality relation does not hold in the case where $A$ is not regular. Hence, for deciding controllability in this case we must proceed differently than in the regular case. By Lemma \ref{lem:contr_and_0contr} one can separate the system into the regular part and the nilpotent part. Let us now focus on a system where $A$ is nilpotent and thus $A^n=0$. Recall that in the observability case with $A$ nilpotent, the only case where observability can hold is the trivial case where no dropout is possible, see the end of the proof of Theorem \ref{t:mainres1}.

\begin{proposition}\label{prop-contr-nilpotent}
Suppose that the matrix $A$ is nilpotent.  The controllability of the system $\sysc$ can be checked algorithmically.
\end{proposition}
\begin{proof}
Let $\ell$ be the smallest $k \in \n_{\leq n} $ such that $A^k=0.$
Hence, by Proposition~\ref{p:reach_rank} we only have to check the rank of  all reachability matrices $C_\sigma (t)$ with $t \in \nat_{\leq \ell}$, which is a finite check over all finite words in $\cW(\cA)$ with maximal length $\ell$.
%
%
%
\end{proof}
	The main result of this section is given next.
	\begin{theorem}
	\label{t:mainres2}
		Problem \ref{p:mainprob2} is decidable in the sense that given a system specified by $\sysc$, there is an algorithm that decides (un)controllability, (un)reachability and (non-)$0$-controllability of the system in finite time.
	\end{theorem}
	
	\begin{proof} If $A$ is regular, one can check reachability by using Proposition \ref{p:obs_and_contr} above, by checking observability for the dual system.  If $A$ is singular, the system $(A,B,\cA)$ is reachable if and only if both the ``regular'' part $(A_r,B_r,\cA)$ and the ``singular'' part $(A_s,B_s,\cA)$ are reachable (this can be seen most easily by inspecting the structure in the reachability matrices based on the block diagonal structure in the real Jordan form).  The first system can be checked with Proposition \ref{p:obs_and_contr} again, and the second one with Proposition \ref{prop-contr-nilpotent}. Hence, reachability is decidable and thus controllability as well (recall that controllability is equivalent to reachability (Proposition \ref{p:reach_and_contr})).\\
	Finally, $0$-controllability (or the absence of it) can be checked by Proposition \ref{p:reach_and_construc}
	(note that the dual of the dual system is the original system) as constructibility is decidable due to Theorem~\ref{t:mainres1}.	\end{proof}
	
	\begin{corollary}
	\label{cor-practical-con}
				The system $\sysc$ is practically reachable (resp. practically controllable and practically $0$-controllable) if and only if it is reachable (resp. controllable and $0$-controllable).
	\end{corollary}
		\begin{proof}
	In the spirit of Proposition~\ref{p:reach_rank} it is obvious that practical reachability is equivalent to the existence of a finite  $T\in\nat$ such that for all $\sigma \in \lA$ it holds that for some $t\in \nat_{\leq T}$ $C_\sigma(t)$ has full rank. From this it follows that $\sysc$ is practically reachable if and only if  $(A_r,B_r,\cA)$ \mh{is practically reachable by Lemma~\ref{lem:inf_to_fin}.}
	 Finally, due to nilpotence of $A_s$ it is clear that $(A_s,B_s,\cA)$ can only be reachable if and only if it is practically reachable. This shows that reachability and practical reachability are equivalent as reachability of a system is equivalent to reachability of its ``regular'' part $(A_r,B_r,\cA)$ and its ``singular'' part $(A_s,B_s,\cA)$.
	
\rmjjj{Now, by the same arguments as in Proposition~\ref{p:reach_and_contr}, practical controllability and practical reachability are equivalent (and we already knew that controllability and reachability are equivalent).
	
	Finally, d}ue to nilpotence of $A_s$ and the results in  Lemma~\ref{lem:contr_and_0contr} it follows that $\sysc$ is practically 0-controllable $\Leftrightarrow$ ${(A_r,B_r,\cA)}$ is practically $0$-controllable $\Leftrightarrow$ ${(A_r,B_r,\cA)}$ is $0$-controllable $\Leftrightarrow$  $\syso$ is  0-controllable. 	
	\end{proof}

	{The relationships among various properties presented so far are summarized in Figure \ref{fig:connxn}.}
\begin{figure*}[htbp]
    \begin{center}
       \includegraphics[height=6cm]{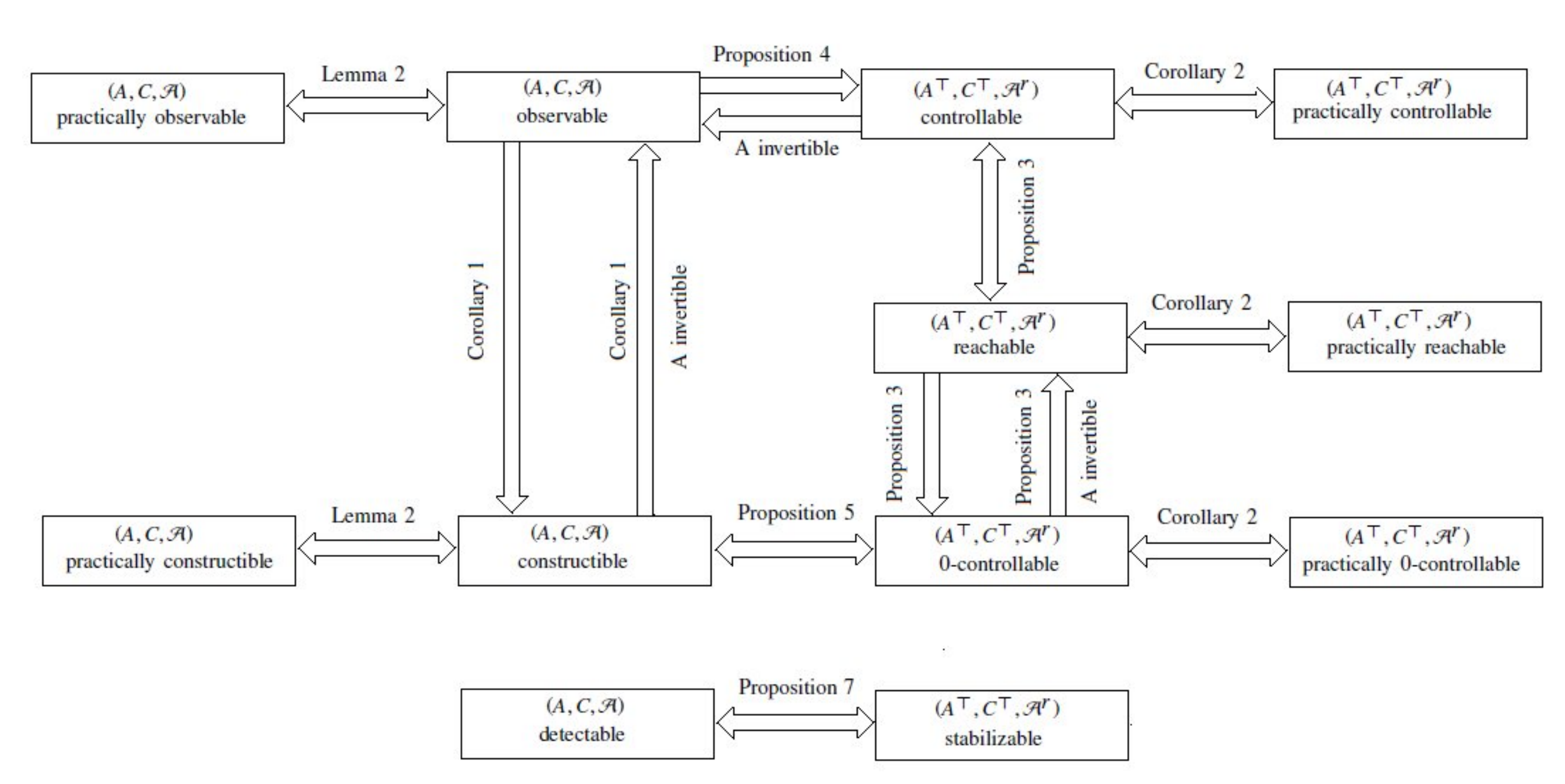}
    \end{center}	
    \caption{Relations among various properties.}  \label{fig:connxn}
    \end{figure*}	

\section{Stabilizability and detectability results}
\label{s:stab_detec_res}
	The next problem that we consider in this paper is given as follows.
	\begin{prob}
	\label{p:mainprob3}
		(a) Given a system $\syso$, determine whether the system is (un)detectable.\\
		(b) Given a system $\sysc$, determine whether the system is (un)stabilizable.
	\end{prob}
	
	\begin{lemma}
	\label{lem:det_condn}
	Consider matrices $A$, $B$, $C$ as in the systems $\syso$ and $ \sysc$. Let $W$ be a similarity transformation matrix, so that in block representation 
		\[
			WAW^{-1} = \pmat{A_{11} & 0\\0 & A_{22}},\ WB = \pmat{B_1 \\ B_2 },\ CW^{-1} = \pmat{C_{1} & C_2},
		\]
		where $A_{11}$ has all the eigenvalues with modulus larger or equal to one, and $A_{22}$ has all the eigenvalues with modulus smaller than one. Then the system $\syso$ is detectable if and only if
	$(A_{11},C_{1},\cA)$ is observable. The system $ \sysc$ is stabilizable if and only if
	$(A_{11},B_{1},\cA)$ is reachable.
	\end{lemma} The lemma is an easy generalization of the LTI case without \akc{data} losses, just like in Lemmas \ref{lem:obsv_and_construc} and \ref{lem:contr_and_0contr}, and we skip the proof.

	An immediate consequence (as observability and reachability were proven to be decidable properties) is the following theorem.
	
	\begin{theorem}
	\label{t:mainres3}
		Problem \ref{p:mainprob3}(a) (resp. \ref{p:mainprob3}(b)) are decidable in the sense that given a system specified by $\syso$ (resp. $\sysc$), there is an algorithm that decides (un)detectability (resp. stabilizability) of the system in finite time.
	\end{theorem}	
	
	We also have the following duality result.
		\begin{proposition}
	\label{p:det_stab_dual}
		The system $\syso$ is detectable if and only if the dual system $(A^\top, C^\top, \cA^{r})$ is stabilizable, where $\cA^{r}$ is the reverse automaton of $\cA$.
	\end{proposition}
	\begin{proof}
	\rmjjj{First, $\syso$ is detectable if and only if} $(A_{11},C_{1},\cA)$ is observable with the matrices $A_{11}$ and $C_1$ as in Lemma~\ref{lem:det_condn} (and transformation matrix $W$). Obviously, $A_{11}$ is  a regular matrix and thus observability of $(A_{11},C_{1},\cA)$ is equivalent to reachability of its dual system $(A_{11}^\top,C_{1}^\top,\cA^r)$. As the latter is equivalent to $(A^\top, C^\top, \cA^{r})$ being stabilizable based on Lemma~\ref{lem:det_condn} (with transformation matrix $W^{-T}$), the proof is complete.
	\end{proof}
	
%
	
	%

\section{Connections with systems subject to varying delays}
\label{s:swdelay_connxn}
In this section we make two links between our problem and a similar one, where the non-idealities in the feedback loop \rmjjj{incur time-varying delays, rather than pure packet losses.  We mainly show that this other model is a particular case of ours, and provide an algorithm to encapsulate any given controllability problem with time varying delays as a controllability problem with packet losses as introduced in this paper (Subsection \ref{subsec:delays2dropouts}).  Then, in Subsection \ref{subsec:dropouts2delays}, we study a particular class of systems from the model studied in this paper, namely, the systems with a bound on the maximal number of consecutive dropouts, and we show how to model these particular systems as systems with varying delays.  The interest of this last result is that alternative algorithms, previously developed for systems with time-varying delays, can then be used for this particular family of systems with packet dropouts.}

\subsection{Controllability under switching delays}

We first briefly recall the framework of  \cite{JUNGERS_12,jungers2014further}, where the problem is to control a linear plant, when the feedback loop is subject to  time-varying delays.  In \cite{JUNGERS_12,jungers2014further} the constraint on the combinatorial structure of the delays is not given in terms of an automaton, but simply in terms of a set of natural numbers, that represent delays, which can possibly be incurred on the feedback signal. 

\begin{definition}
  A \emph{system with varying delays} is defined by a triple $(A,B,\cD)$ with $A\in \ree^{n\times n}$, $B \in \ree^{n\times m}$ and $\cD=\{d_1,d_2,\dots,d_N\}\subset \n$ is the set of \emph{admissible delays}. Its dynamics is described by the equations
\begin{equation}\label{system-varying-delays}
x(t+1)=
Ax(t) + \sum_{t'\leq t, t'+d(t')=t}B u(t')
\end{equation}
with  $$d: \n \rightarrow \cD$$ being any arbitrary \emph{delay signal}.
\end{definition}
The meaning of Equation \eqref{system-varying-delays} is that at each time $t'=1,2,\dots,$ the control packet is subject to a certain delay $d(t'),$ which is determined by an exogenous switching signal $d.$  The control packet $u(t')$ will impact the state of the system not at time $t'+1$ as in a classical LTI system, but at time $t'+d(t')+1.$  If several packets arrive at the actuator at the same time, it was arbitrated in \cite{JUNGERS_12,jungers2014further} that they are simply added.\\
The trajectory generated by the system \eqref{system-varying-delays} with initial state $x(0)=x_0$, input sequence $u:\nat \rightarrow \ree^m$ and delay  signal  \rmjjj{$d: \nat \rightarrow \cD$} is denoted by $x_{x_0,d,u}$.

In the  spirit of Definition \ref{d:all-defn2} {the following definition was introduced in \cite{jungers2014further}}.

\begin{definition} \label{defi-contr-delays} We say that a system with varying delays described by the triple  $(A,B,\cD)$ is \emph{controllable} if for any delay signal $d:\ \n \rightarrow \cD,$  any initial state $x_0\in \ree^n$ and any final state $x_f\in\ree^n$, there is an input signal $u:\nat\rightarrow \ree^m $ such that $x_{x_0,d,u}(T)=x_f$ for some $T\in \n$.  In case the system is not controllable it is  said to be \emph{uncontrollable.}
\end{definition}
In this setting too, given a delay signal $d:\nat\rightarrow \cD$ one can introduce the controllability matrix $\bar C_d(t)$ at time $t\in\nat$.

\begin{definition}\label{def-contr-matrix-delays} Consider a system with time-varying delays \eqref{system-varying-delays}.
Given $d:\nat\rightarrow \nat$ we define the \emph{controllability matrix} at time $t\in\nat$ as  $\bar C_d(t)\in \ree^{n\times mt}$, whose $i$-th block-column (for all $i\in \nat_{[1,t]}$) is given by
\begin{itemize} \item $A^{(t-i-d(i-1))}B$,  if $i+d(i-1)\leq t;$
\item the zero column, if $i+d(i-1)>t.$
\end{itemize}
\end{definition}

The state of a system with varying delays as in \eqref{system-varying-delays} can be expressed thanks to its controllability matrix as
\begin{equation} x_{x_0,\sigma,u}(t) = A^t x_0 + \bar C_d(t) u^{t-1}, \label{eq:state_evol_delays} \end{equation}
where $u^{t-1}\in \ree^{mt}$ is the column vector with all the inputs $u(s),$ $s=0,\dots, t-1.$  Again, controllability can be related to the rank of the matrix $\bar C_d(t)$ as formalized next.
\begin{proposition}\label{prop-contr-delays} {[adapted from Proposition 4 in \cite{jungers2014further}] }
The delay system given by $(A,B,\cD)$  is uncontrollable if and only if there is a delay signal $d:\nat \rightarrow \cD$ such that for all $ t\in\nat$ it holds that $\bar C_d(t)$ does not have full rank.
\end{proposition}
\rmjjj{In order to best represent the link between the varying delays model and our packet dropouts model, we introduce another useful concept: Given a delay signal $d:\nat\rightarrow \nat$, we define the \emph{actuation signal $\tau_d: \; \n\rightarrow \{0,1\}$} for $t\in\nat$ by
\begin{equation} \label{eq-actuation-signal}\tau_d(t)=\begin{cases} 1, & \text{ if there is a } t' \in \nat \text{ such that } t'+d(t')=t, \\
 0, & \text{ otherwise.} \end{cases}\end{equation}
With this definition in place, the link between the two models is easy to express: it turns out that for all $d:\nat \rightarrow \nat$ and all $t\in \nat,$ it holds that \begin{equation}\label{eq-link}\imag C_{\tau_d}(t) = \imag \bar C_d(t),\end{equation} where $ C_{\tau_d}(t)$ is the Reachability matrix as defined in \ref{d:reach_mat}.}

Proposition \ref{prop-contr-delays} gives us a characterization of controllable systems, but again not an algorithmic one. Indeed it is not clear how to generate the signal $\sigma$
 as in the proposition, and even less clear how to prove the \rmj{in}existence of such an infinite-length signal.  However, it is shown in \cite{jungers2014further} that this problem is decidable and even more, that one can bound the time needed by the algorithm to answer the question.

\subsection{The switching delays problem expressed as a dropout problem}  \label{subsec:delays2dropouts}

In this subsection, we prove that the setting studied in this paper (where an automaton describes the possible packet loss signals) is \emph{more general than the model with switching delays}.  Indeed, we take an arbitrary controllability problem with switching delays, and we convert it to a controllability problem with a dropout signal constrained by an automaton, as studied in this paper.  Thus, the controllability test introduced in this paper can be applied to the previously studied problem of controllability with {varying delays.} However, the next example demonstrates that the proposed results in this paper apply to a much richer class of automata than the ones produced by varying delays.


\begin{example}
Consider the automaton 
described by \(M = \pmat{0 & 1 & 0\\0 & 0 & 1\\1 & 0 & 0}\) and \(s = \pmat{1\\1\\0}\), which is depicted in Figure \ref{fig:delay_ex}.
\begin{figure}[h!]
\begin{center}
        \begin{tikzpicture}[every path/.style={>=latex},every node/.style={auto}]
            \node[state]            (b) at (1,-1)  { $s_3=0$ };
            \node[state]            (c) at (-1,1) { $s_1=1 $};
            \node[state]            (d) at (1,1)  {$s_2=1$};

            \draw[->] (b) edge (c);
            \draw[->] (d) edge (b);
             \draw[->] (c) edge (d);
        \end{tikzpicture}
        \end{center}
\caption{An automaton that cannot be simulated by a system with varying delays.}
\label{fig:delay_ex}
\end{figure}
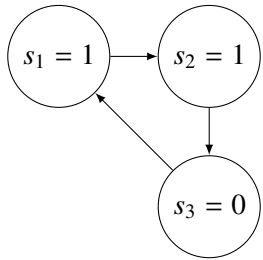
Clearly, \(\lA =\{ (0, 1, 1, 0 , 1 , 1, 0, 1, 1,\ldots), (1, 1, 0 , 1 , 1, 0, 1, 1,0,1,1, \ldots),\) \((1, 0 , 1 , 1, 0, 1, 1,0,1,1, \ldots)\}\).  We claim that there is no system with varying delays that generates a language of actuation signals equal to $\lA .$  To see this, observe that if the delay signal is constant (i.e., the delay does not change with time), the actuation signal is asymptotically equal to $(1,1,1,\dots),$ which is obviously not in $\lA .$
Due to this structure this automaton cannot be captured as a varying delay system.
\end{example}

\begin{theorem}
Let $(A,B,\cD)$ represent a system as in Definition~\ref{def-contr-matrix-delays}.  One can construct an automaton $\cA$ with the following property: the system $( A,B,\cA)$ is controllable (in the sense of Definition \ref{d:all-defn2}) if and only if the system \((A,B,\cD)\) is controllable (in the sense of Definition \ref{defi-contr-delays}).
\end{theorem}
\begin{proof}
We have seen in Proposition \ref{prop-contr-delays} that the uncontrollability of system \eqref{system-varying-delays} boils down to the existence of a delay signal $d:\nat\rightarrow\cD$ such that the controllability matrix $\bar C_d(t)$ has rank smaller than $n$ for all $t\in\nat.$   That is, we want to know whether for all $t,$  $\imag\bar C_d(t)$, which is equal to
\begin{equation}\label{eq-transfo}
\imag\{A^{t-t'-d(t')-1}B:\, t'\in\nat,\ t'+d(t')+1 \leq t\},
\end{equation} is not equal to $\ree^n$.  \rmj{This latter equation bears similarity with our definition \eqref{e:reach_mat} of reachability matrix, and suggests to build a dropout system $(A,B,\cA),$ such that the possible controllability matrices have block-colums as in \eqref{eq-transfo}.}  \rmjjj{More precisely, we will build an automaton $\cA$ such that $\cL(\cA)$ is almost equal to the set of possible actuation signals $\tau_d$ (they only differ at a few initial letters, but we show below that nevertheless, one system is controllable if and only if the other is). 

We now build this automaton. }It corresponds to the so-called \emph{De Bruijn graph} of the sequences on the alphabet $\cD$, which can be constructed as follows \rmjjj{(we note $d_{max}$ for the largest delay in $\cD$)}. 
The nodes correspond to the vectors in $\cD^{d_{max}+1},$ and for each node $v=(v_0,\ldots,v_{d_{max}}),$ and each delay $d\in \cD,$ \rmj{there is an edge labeled with $d$ pointing from $v$ to the node $v'=(d,v_0,\dots,\bar v_{d_{max}-1}).$}  Each \rmjjj{node} represents the last values of the delay signal at a certain time  and from this information one can easily compute if the {actuation} signal at this node is equal to zero or one.  More precisely:
\begin{equation}
\label{e:node_label_bruijn} s_v =1  \text{ iff }  \exists i\in \nat_{[0,\dmax]} \text{ with }v_i=i.\end{equation}
%

It remains to prove that the constructed system $\sysc$ is controllable if and only if the  delay system  \((A,B,\cD)\)  is controllable.  
{\bf If:} Let us take such an unobservable delay signal $d:\ \n \rightarrow \cD.$  We start from the vertex corresponding to the vector $(d(d_{max}),d(d_{max}-1),\ldots,d(0)),$ and then follow the path with edges labeled $d(\dmax +1),d(\dmax +2),\ldots.$
Hence, the node $v(t)$ in the De Bruijn graph $\cA$ at time  $t\in\nat$ is given by $v(t) = (d(t+d_{max}),d(t+ d_{max}-1),\ldots,d(t)),$ and
for any $t\geq 0,$ the corresponding packet loss signal is given by $ \sigma(t)=\tau_d(t+d_{max})$ with $\tau_d$ the actuation signal corresponding to $d$ (see \eqref{eq-actuation-signal}).  Thus (see \eqref{e:node_label_bruijn}),
$$\sigma(t)=1 \Leftrightarrow \exists s \in \nat_{[t,\dmax+t]}: \, d(s)=t+\dmax-s,$$
{and from \eqref{e:reach_mat}, the image of the reachability matrix
$C_\sigma(t)$ at an arbitrary time $t\in\nat$ is included in }

\begin{equation} \label{temp1}
\imag\{A^{t-t'-d(t')-1+d_{max}}B: t\leq t' \leq t+d_{max}\}.
\end{equation}
This set is exactly $A^{\dmax}\imag(\bar C_{d}(t))$ with $\bar C_{d}(t)$ is the controllability matrix determining the controllability of our system with varying delays (see \eqref{eq-transfo}), which is never of rank $n$ by hypothesis. Hence, we have found a packet loss signal $\sigma: \nat \rightarrow \{0,1\}$ in $\lA$ such that the reachability matrix has never full rank and thus the system is not controllable.\\
{\bf Only if:}
Conversely, if there is a data loss signal $\sigma\in\lA$ such that the corresponding reachability matrix $C_\sigma(t)$ is never of rank $n,$ the corresponding delay signal $d$ obtained by the successive labels $d_{i_1},\dots,d_{i_t}$ on the edges of $\sigma$ leads to a controllability matrix (see \eqref{eq-transfo} again) with
$$ \imag\bar C_d(t)=\imag\{A^{t-t'-d(t')-1}B:\, t'\geq 0,\ t-t'-d(t')-1\geq 0\}.$$ All these {colums} in this expression are colums of the reachability matrix \eqref{e:reach_mat} corresponding to the path $\sigma$ (see our definition of the labeling of the nodes of the automaton as in \eqref{e:node_label_bruijn}), and so the matrix cannot be of rank $n.$
%
\end{proof}

We conclude this subsection by noting that this result gives an algorithm for deciding controllability of systems with varying delays: Simply construct the automaton $\cA$ described in the above proof, and decide \rmj{controllability of the system $(A,B,\cA)$} with the techniques presented in this paper.
We note, however, that the alternative technique presented in \cite{jdd14} seems more advantageous, since: 1. It does not require to build an auxiliary automaton, whose size can be exponential in the size of the initial problem and 2. in \cite{jdd14}, the total computational time of the algorithm is bounded w.r.t. the size of the initial \rmj{problem, as a function that does not depend on the numerical values of the entries of $A.$ Thus, the algorithm developed here might not be competitive in terms of computational time with the one previously introduced}.  

\rmj{\subsection{The \rmjjj{particular case of maximal consecutive dropouts}} \label{subsec:dropouts2delays} As we have just seen in the previous subsection, the framework of this paper (where nonidealities are represented by an automaton) is more general than systems with switching delays as in \eqref{system-varying-delays}.  However, in this subsection we study the particular case where the automaton represents the constraint that the actuation signal cannot contain more than $N$ consecutive dropouts (like in Example \ref{ex-maxdropouts}).
We show that in this particular case, the controllability problem can be restated as a controllability problem with switching delays. That is, the switching delays model encapsulates the constraint of maximal consecutive dropouts.}  As a consequence, controllability with a maximum number of dropouts in a row can be solved with the methodology developed there, for which a better bound on the running time  of the algorithm is available. 
%
However, we attract the attention of the reader on the fact that the techniques developed in \cite{jungers2014further} can only be applied for the very particular case where the automaton is of the form in Fig. \ref{fig:automat_ex} (i.e., when the constraint on the data loss signal is a bound on the number of consecutive losses).  The case with an arbitrary automaton is much more general and powerful (in terms of modeling of a data losses signal).

 \begin{theorem}
Let us consider a system $(A,B,\cA)$ with at most $N$ consecutive dropouts  (like defined in Example \ref{ex-maxdropouts}) for some $N\in \nat.$  The system $(A,b,\cA)$ is controllable if and only if the delay system $(A,B,\cD)$ with $\cD=\{0,1,\dots,N\},$ is controllable.
\end{theorem}
\begin{proof}First, observe that in both Definitions \ref{d:reach_mat} and \ref{def-contr-matrix-delays}, the columns appearing in the controllability matrix are uniquely defined by the packet loss signal $\sigma$ (respectively the actuation signal $\tau_d$), which characterizes the set of times where  $B$ enters the controllability matrix. 

 Hence, if we can show that the language of actuation signals $\sigma$ in the max-$N$-consecutive-dropouts case, and the language  of actuation signals $\tau_d$ induced by delay signals $d:\nat \rightarrow \cD$ in the varying delay case with $\cD=\{0,1,\dots,N\}$ are the same, we have established the proof of the theorem.  Indeed, there would then exist an actuation signal $\sigma$ which is uncontrollable in the max-$N$-consecutive-dropouts case if and only if there exists an uncontrollable delay signal $d:\nat \rightarrow \cD$ in the switching delays case.

To show that the languages are indeed equal, first take any signal $\sigma:\nat \rightarrow \{0,1\}$ satisfying the max-$N$-consecutive-dropouts constraint, that is
\begin{equation}\label{eq-construction-equiv}\forall t\in\nat, \exists t'\in \nat_{[t,t+N]} \text{ such that }  \sigma(t')=1.\end{equation}
The signal $\sigma$ is  the same as  $\tau_d$ corresponding to the delay signal $d:\nat \rightarrow \nat$ given for $t\in\nat$  by $$ d(t)=\min_{t'\geq t, \sigma(t')=1}(t'-t).$$  By  \eqref{eq-construction-equiv}, the above defined signal $d(\cdot{})$ is admissible in the sense that $d:\nat\rightarrow D$ with   $D=\{0,1,\dots,N\}$, as the right-hand side in the equation above is always smaller or equal to $N.$

Second, take a signal $d:\, \n \rightarrow \{0,1,\dots,N\}.$ It obviously incurs an actuation signal $\tau_d$ with at most $N$ zeros in a row, which is clearly contained in the set of admissible signals $\sigma$ for the max-$N$-consecutive-dropout case.  
\end{proof}

\section{Numerical example}
\label{s:num_ex}
    Consider a discrete-time linear system subject to data losses \eqref{e:swsys1} with
    \[
        A = \pmat{-2 & -13 & 9\\-5 & -10 & 9\\-10 & -11 & 12}\quad\text{and}\quad C = \pmat{1 & 2 & 3}.
    \]
    The observability matrix in the standard sense (with \(\sigma(t) = 1\) for all \(t\in\nat\)) of the discrete-time linear system is
    \[
        O(A,C) = \pmat{C \\ CA \\CA^2} = \pmat{1 & 2 & 3\\-42 & -66 & 63\\-216 & 513 & -216},
    \]
    the rank of which is 3. Consequently, the pair \((A,C)\) is observable in the standard sense.

    Now consider a data loss signal \(\sigma\) such that maximum two consecutive data losses are allowed. The corresponding automaton \(\cA = (M,s)\) with
    \[
        M = \pmat{0 & 1 & 1\\0 & 0 & 1\\1 & 0 & 1}\quad s = \pmat{0\\0\\1}
    \]
    is shown in Figure \ref{fig:num_automaton}.
    \begin{figure}[h!]
     \begin{center}
        \begin{tikzpicture}[every path/.style={>=latex},every node/.style={auto}]
            \node[state]            (b) at (-1,1)  { $s_1=0$ };
            \node[state]            (a) at (1,0)  { $s_2=0$ };
            \node[state]            (c) at (-1,-1) { $s_3=1$ };

            \draw[->] (c) edge[bend left] (b);
            \draw[->] (b) edge[bend left] (c);
            \draw[->] (b) edge (a);
            \draw[->] (a) edge (c);
            \draw[->] (c) edge[loop left] (c);
        \end{tikzpicture}
        \end{center}
        \caption{Automaton representing data loss signals with at most 2 consecutive losses.} \label{fig:num_automaton}
     \end{figure}
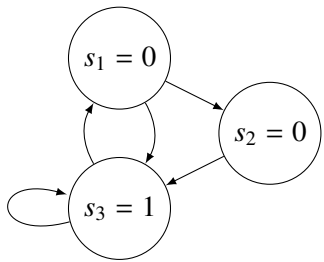

    The problem under consideration is to decide whether the system \((A,C,\cA)\) is (un)observable. To this end, our algorithm checks unobservability of periodic signals corresponding to cycles of increasing length. It is easy to verify (once the signal has been found) that the system \((A,C,\cA)\) is unobservable under the periodic data loss signal \(\sigma = (1,0,0,1,0,0,1,0,0,1,\ldots)\).

    Now let us modify the constraint on the data loss signal. Consider the case where the control and communication system is constructed so as to protect itself against disruptions, for example, such that at most one packet can be lost in a period of $m$ consecutive steps (or stated otherwise, at least $m-1$ successful transmissions take place in every $m$ steps, which is $(m,m-1)$-firmness \cite{horssen_ecc16}. 
    It turns out that in this case the system is observable (as soon as $m\geq 2$).  Indeed, for $m=2,$ the automaton \(\cA = (M,s)\) capturing the above constraint is given below \cite{horssen_ecc16}. 
    \[
        M = \pmat{0 & 1\\1 & 1}\quad\text{and}\quad s = \pmat{0\\1}.
    \]
    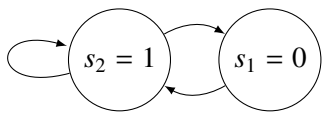
\begin{figure}[h!]
     \begin{center}
        \begin{tikzpicture}[every path/.style={>=latex},every node/.style={auto}]
            \node[state]            (b) at (1,1)  { \(s_{1} = 0\) };
            \node[state]            (c) at (-1,1) { \(s_{2} = 1\) };

            \draw[->] (c) edge[bend left] (b);
            \draw[->] (b) edge[bend left] (c);
            \draw[->] (c) edge[loop left] (c);
        \end{tikzpicture}
        \end{center}
        \caption{Automaton representing data loss signals with at most 1 loss at a time.} \label{fig:num_automaton2}
     \end{figure}
It is easy to see that the system is observable, because every finite signal of length $6$ leads to a full rank observability matrix.  Clearly, if the system is observable with $m=2,$ it is observable for any $m\in \nat_{>2}.$

		This leads to the question of how fast our algorithm can decide (un)observability for a given \((A,C,\cA)\).
		Recall that the proposed algorithm checks (un)observability of each periodic signal corresponding to cycles of length \(s\), \(s = 1,\ldots,PN\). If all these cycles lead to observable signals, then the system is observable. However, in practice, it may be possible to conclude that a system is observable much earlier than checking all cycles up to length \(PN\) . In particular, it may be possible to conclude about (un)observability of a cycle from the properties of other cycles, which reduces the total number of signals for which the (un)observability is checked. For example, in the case presented above, the observability of the signal \(\sigma = (1,0,1,0,1,\ldots)\) leads to observability of the signal \(\sigma = (0,1,0,1,0,1,\ldots)\). Research towards reducing the number of switching signals for which (un)observability needs to be checked, is currently an open direction.
		
\section{Conclusion}
\label{s:concln}
{In this paper, our goal was to provide a theoretical analysis of how classical control-theoretic notions are modified by typical non-idealities that are more and more present in modern control, as for instance in wireless control networks.  From a formal point of view, the problems we are trying to solve are extremely hard, and even undecidable for arbitrary switched systems.  However, our systems, while formally falling in the class of switched systems, are not arbitrary, and their algebraic structure actually allows to retrieve (part of) the powerful properties that are so useful for classical LTI systems. We have shown that packet dropouts (under constraints represented by an automaton), while formally transforming the systems into switched systems, keep enough algebraic structure so that fundamental questions like controllability or observability can still be answered.
As we have seen, the non-idealities under study in this paper are quite general, and they even encapsulate previously studied models, like switching delays.

In further work, we plan to study models where observability and controllability interact in the same feedback loop.  Also, concerning the controllability problem, our work leaves natural open questions; for instance, how to choose the control signal in real time?  Indeed, remark that, in cases where our algorithm returns a positive answer, we know that \emph{there exists} a control signal that allows to stabilize the plant, but nothing is said about how to chose the optimal signal.  It is clear from our results how to chose this signal \emph{a posteriori}, i.e., once we know the dropout signal that effectively happened.  However, in several practical situations, it does not make sense to assume that the disruptions in the feedback loop are known beforehand, and thus the controller has to `guess' the optimal control value independently of the dropout signal.  From this point of view, our results can be viewed as necessary conditions for controllability (or stabilizability, or...), which are not sufficient in the case of real-time control (note that the observability problem does not suffer this limitation, since the essence of observability is that one recovers the state a posteriori).}
	
\bibliographystyle{siam}
\bibliography{refr_dropout}
\end{document}